\documentclass[11pt]{article}%
\usepackage{amsmath}
\usepackage{amssymb}
\usepackage{amsfonts}
\usepackage{graphicx}%
\providecommand{\U}[1]{\protect\rule{.1in}{.1in}}
\providecommand{\U}[1]{\protect\rule{.1in}{.1in}}
\newtheorem{theorem}{Theorem}

\newtheorem{definition}[theorem]{Definition}

\newtheorem{proposition}[theorem]{Proposition}

\newenvironment{proof}[1][Proof]{\noindent\textbf{#1.} }{\ \rule{0.5em}{0.5em}}
\begin{document}

\title{Automorphism Groups of Cyclic $p$-gonal Pseudo-real Riemann Surfaces}
\author{Emilio Bujalance\thanks{ Partially supported by MTM2014-55812.}\\Departamento de Matem\'{a}ticas Fundamentales\\Facultad de Ciencias, UNED\\Senda del rey, 9, 28040 Madrid (Spain)\\eb@mat.uned.es
\and Antonio F. Costa$^{\ast}$\\Departamento de Matem\'{a}ticas Fundamentales\\Facultad de Ciencias, UNED\\Senda del rey, 9, 28040 Madrid (Spain)\\acosta@mat.uned.es}
\date{}
\maketitle

\textbf{Abstract.} In this article we prove that the full automorphism group
of a cyclic $p$-gonal pseudo-real Riemann surface of genus $g$ is either a
semidirect product $C_{n}\ltimes C_{p}$ or a cyclic group, where $p$ is a
prime $>2$ and $g>(p-1)^{2}$. We obtain necessary and sufficient conditions
for the existence of a cyclic $p$-gonal pseudo-real Riemann surface with
full\ automorphism group isomorphic to a given finite group. Finally we
describe some families of cyclic $p$-gonal pseudo-real Riemann surfaces where
the order of the full automorphism group is maximal and show that such
families determine some real 2-manifolds embbeded in the branch locus of
moduli space.

{\small Submitted 7/07/14. This is a revised version 13/03/15.}



\section{Introduction}

A Riemann surface is called \textit{pseudo-real}\emph{\/ }if it admits
anticonformal automorphisms but no anticonformal involution. Pseudo-real
Riemann surfaces appear in a natural way in the study of the moduli space
$\mathcal{M}_{g}^{K}$ of Riemann surfaces considered as Klein surfaces. The
moduli space $\mathcal{M}_{g}$ of Riemann surfaces of genus $g$ is a two-fold
branch covering of $\mathcal{M}_{g}^{K}$ and the preimage of the branch locus
consists of the Riemann surfaces admitting anticonformal automorphisms; they
are either real Riemann surfaces (admitting anticonformal involutions, see for
instance \cite{N}) or pseudo-real Riemann surfaces.

Pseudo-real Riemann surfaces are those Riemann surfaces which are equivalent
to their conjugate but the equivalence is not realized by an involution;
therefore they admit anticonformal automorphims of order greater than $2$. The
square of such automorphisms are not trivial conformal automorphisms, which
means that the points in $\mathcal{M}_{g}$ corresponding to pseudo-real
Riemann surfaces are in the singular set (branch locus) of the orbifold
$\mathcal{M}_{g}=\mathbb{T}_{g}/\mathrm{Mod}_{g}$.

Pseudo-real Riemann surfaces were first studied in \cite{E} and \cite{S}.
Hyperelliptic pseudo-real Riemann surfaces were considered in \cite{S3},
\cite{BCn}, and \cite{BT} and pseudo-real Riemann surfaces with cyclic
automorphism groups in \cite{Et}. In \cite{BCC} the existence of pseudo-real
Riemann surfaces of every genus $g\geq2$ \ is established and this is followed
by a study of pseudo-real Riemann surfaces of genus $2$ and $3$; in \cite{BCs}
all the topological types of actions of the automorphisms groups for these two
genera and genus $4$ are described. In \cite{H1} and \cite{H2} the author
finds explicit equations for non-hyperelliptic pseudo-real Riemann surfaces.

The $p$-gonal surfaces are cyclic $p$-fold coverings of the Riemann sphere and
there is a great deal of interest in the study of the automorphism groups of
these surfaces (see for instance \cite{B}, \cite{BCI}, \cite{BHS}, \cite{Ko},
\cite{W}). The groups of (conformal and anticonformal) automorphisms of
$p$-gonal real Riemann surfaces are obtained in \cite{BCI}, under the
assumption that the $p$-gonal morphism is normal; for the conformal
automorphisms in the non-normal case see \cite{W}.

In the present work, we study the full groups of (conformal and anticonformal)
automorphisms of pseudo-real Riemann surfaces of genus $g$ that are cyclic
$p$-gonal, where $p$ is a prime $>2$ and $g>(p-1)^{2}$. In these conditions we
establish that there are only two possible types of full automorphism groups
of cyclic $p$-gonal pseudo-real Riemann surfaces: they are either cyclic or
semidirect product of cyclic groups $C_{n}\ltimes C_{p}$ (Theorem
\ref{mainthm}). We also obtain necessary and sufficient conditions for the
existence of cyclic $p$-gonal pseudo-real Riemann surfaces with given full
automorphism group.

In section 4 we obtain the maximal order of the automorphism groups of
pseudo-real cyclic $p$-gonal Riemann surfaces of genus $g>(p-1)^{2}$, when
$p-1$ divides $g$.

Finally, in section 5, we apply our results to describe some 2-manifolds
embedded in the branch locus of moduli space corresponding to cyclic $p$-gonal
pseudo-real Riemann surfaces with maximal symmetry.

\textbf{Acknowledgements.} We would like to express our thanks to Marston
Conder for advises about metacyclic groups. The authors are indebted to the
referee for careful reading of this article, many helpful remarks that improve
the quality of the exposition and to point out an error in a preliminar
version of Theorem \ref{mainthm}.

\section{Non-Euclidean crystallographic groups and Pseudo-real Riemann
surfaces}

A \textit{non-Euclidean crystallographic group} (or \textit{NEC group}) is a
discrete group of isometries of the hyperbolic plane $\mathbb{D}$ (we consider
the unit disc model). We shall assume that an NEC group has compact orbit
space. If $\Gamma$ is such a group, its algebraic structure is determined by
its signature%
\begin{equation}
(h;\pm;[m_{1},\ldots,m_{r}];\{(n_{11},\ldots,n_{1s_{1}}),\overset{(k)}{\ldots
},(n_{k1},\ldots,n_{ks_{k}})\}). \label{sign}%
\end{equation}

The orbit space $\mathbb{D}/\Gamma$ is a surface, possibly with boundary. The
number $h$ is called the\textit{\ genus} of $\Gamma$ and equals the
topological genus of $\mathbb{D}/\Gamma$, while $k$ is the number of its
boundary components; the sign is $+$ or $-$ according to whether or not the
surface $\mathbb{D}/\Gamma$ is orientable. The integers $m_{i}\geq2$ are
called the \textit{proper periods}, and represent the branch indices over
interior points of $\mathbb{D}/\Gamma$ in the natural projection
$\pi:\mathbb{D}\rightarrow\mathbb{D}/\Gamma$. The bracketed expressions
$(n_{i1},\ldots,n_{is_{i}})$, some or all of them may be empty (with $s_{i}%
=0$), are called \textit{period cycles} and represent the branchings over the
${i}^{\mathrm{th}}$ hole in the surface, and the numbers $n_{ij}\geq2$ are
named \textit{link periods}.

There exists, associated with each signature, a canonical presentation for the
group $\Gamma$, and a formula for the hyperbolic area of its fundamental
domain (see \cite{BEGG}). If the signature has sign $+$ then $\Gamma$ has the
following generators:

\qquad$x_{1},\ldots,x_{r}$ \ (elliptic transformations),

\qquad$c_{10}, \ldots, c_{1s_{1}}, \ldots, c_{k0}, \ldots, c_{ks_{k}} $ \ (reflections),

\qquad$e_{1},\ldots,e_{k}$ \ (boundary transformations),

\qquad$a_{1},b_{1},\ldots,a_{g},b_{g}$ \ (hyperbolic transformations);

\noindent and these generators satisfy the relations

\qquad$x_{i}^{m_{i}} = 1$ \ (for $1 \leq i \leq r$),

\qquad$c_{ij-1}^{2}=c_{ij}^{2}= (c_{ij-1}c_{ij})^{n_{ij}}=1 , \ c_{is_{i}%
}=e_{i}^{-1}c_{i0}e_{i}$ \ (for $1 \leq i \leq k, 0 \leq j \leq s_{i} ), $

\qquad$x_{1}\ldots x_{r}e_{1}\ldots e_{k}a_{1}b_{1}a_{1}^{-1}b_{1}^{-1}\ldots
a_{h}b_{h}a_{h}^{-1}b_{h}^{-1}=1$ (long relation)

If the sign is $-$ then we just replace the hyperbolic generators $a_{i}%
,b_{i}$ by glide reflections $d_{1},\ldots,d_{h}$, and the long relation by
$x_{1}\ldots x_{r}e_{1}\ldots e_{k}d_{1}^{2}\ldots d_{h}^{2}=1.$

The hyperbolic area of an arbitrary fundamental region of an NEC group
$\Gamma$ with signature (\ref{sign}) is%
\[
\mu(\Gamma)=2\pi\left(  \varepsilon h-2+k+\sum_{i=1}^{r}{\left(  1-{\frac
{1}{m{_{i}}}}\right)  }+{\frac{1}{2}}\sum_{i=1}^{k}\sum_{j=1}^{s_{i}}{\left(
1-{\frac{1}{n{_{ij}}}}\right)  }\right)
\]
where $\varepsilon=2$ if the sign is $+$, and $\varepsilon=1$ if the sign is
$-$. Furthermore, any discrete group $\Lambda$ of isometries of $\mathbb{D}$
containing $\Gamma$ as a subgroup of finite index is also an NEC group, and
the hyperbolic area of a fundamental region for $\Lambda$ is given by the
Riemann-Hurwitz formula:

\begin{center}
$[\Lambda:\Gamma]=\mu(\Gamma)/\mu(\Lambda).$
\end{center}

For any NEC group $\Lambda$, the \textit{canonical Fuchsian subgroup
\/}${\Lambda}^{+}$ is its subgroup of orientation-preserving elements. If
$\Lambda^{+}\neq\Lambda$ then $\Lambda^{+}$ has index $2$ in $\Lambda$, and we
say that $\Lambda$ is a \textit{proper} NEC group.

Let $X$ be a compact Riemann surface of genus $g>1$. Then there is a surface
Fuchsian group $\Gamma$, which is an NEC group with signature
$(g;+;[-];\{-\})$, such that $X=\mathbb{D}/\Gamma$, and the full (conformal
and anticonformal) automorphism group Aut$(X)$ of $X$ is isomorphic to
$\Delta/\Gamma$, where $\Delta$ is an NEC group normalizing $\Gamma$. We
denote by Aut$^{+}(X)$ the group $\Delta^{+}/\Gamma$ of all conformal
automorphisms of $X$.

\begin{definition}
A Riemann surface is called \textrm{pseudo-real \/}if it admits anticonformal
automorphisms but no anticonformal involution.
\end{definition}

In \cite{BCC} we have established some basic results on the automorphism
groups of pseudo-real Riemann surfaces

\begin{proposition}
\cite{BCC}\label{4divides} Let $X$ be a pseudo-real Riemann surface, and let
$G$ be the full automorphism group of $X$. Then $4$ divides the order of $G$.
\end{proposition}

\begin{proposition}
\cite{BCC}\label{signatureDelta} Suppose the pseudo-real surface $X$ is
conformally equivalent to $\mathbb{D}/\Gamma$, where $\Gamma$ is a surface
Fuchsian group which is normalized by an NEC group $\Delta$ such that
$\Delta/\Gamma\cong G=\mathrm{Aut}(X)$. Then the signature of $\Delta$ has the
form $(h;-;[m_{1},...,m_{r}])$ and the signature of the canonical Fuchsian
subgroup $\Delta^{+}$ of $\Delta$ is
\[
(h-1;+;[m_{1},m_{1},m_{2},m_{2},...,m_{r},m_{r}]).
\]

\end{proposition}

\section{Automorphism groups of cyclic $p$-gonal pseudo-real Riemann surfaces}

In this section we establish our main theorem on the algebraic structure of
the full automorphism groups of cyclic $p$-gonal pseudo-real Riemann surfaces:

\begin{theorem}
\label{mainthm} Let $X$ be a cyclic $p$-gonal ($p$ is a prime $>2$)
pseudo-real Riemann surface of genus $g>(p-1)^{2}$. Let $G$ be the full
automorphism group of $X$, $H\cong C_{p}$ be the subgroup of $G$ generated by
a $p$-gonality automorphism and $\Delta$ be an NEC\ group uniformizing $X/G$.
Hence the genus $g$ of $X$ must be even and the possible isomorphy types for
the group $G$ are:

1. $G\cong C_{n}\ltimes_{r}H$, where $4$ divides $n$, the first factor is
generated by an anticonformal automorphism and
\begin{equation}
C_{n}\ltimes_{r}H=\left\langle x,y:x^{p}=1;y^{n}=1;y^{-1}xy=x^{r}\right\rangle
\label{PresenMeta}%
\end{equation}
where $0<r<p$, $r^{n}\equiv1\operatorname{mod}p$. The NEC\ group $\Delta$ has
signature either:%
\begin{equation}
(1;-;[p,\overset{\frac{2(g+p-1)}{n(p-1)}}{...},p,n/2]) \label{signature1}%
\end{equation}
or, if in the presentation (\ref{PresenMeta}) $r=1$ (direct product) or
$r=p-1$, the signature of $\Delta$ may be also:%
\begin{equation}
(1;-;[p,\overset{\frac{2g}{n(p-1)}}{...},p,\frac{n}{2}p]) \label{signature2}%
\end{equation}

2. $G\cong C_{np}$ (where $4$ divides $n$) and $G$ is generated by an
anticonformal automorphism. The NEC\ group $\Delta$ has either signature
(\ref{signature1}) or (\ref{signature2}).
\end{theorem}

\begin{proof}
Let $\Gamma$ be a surface Fuchsian group uniformizing $X$, i. e.
$X=\mathbb{D}/\Gamma$. There is an NEC\ group $\Delta$ and a Fuchsian group
$\Lambda$ such that $X/G\cong\mathbb{D}/\Delta$, $X/H\cong\mathbb{D}/\Lambda$
and
\[
\Gamma\vartriangleleft\Lambda<\Delta;\Gamma\vartriangleleft\Delta
;\Lambda/\Gamma\cong H;\Delta/\Gamma\cong G
\]
The signature of $\Lambda$ is%
\[
(0;+;[p,\overset{q}{...},p])
\]
where $q=\frac{2(g+p-1)p}{p-1}$, and by Proposition \ref{signatureDelta}, the
signature of $\Delta$ has the form
\[
(h;-;[m_{1},...,m_{r}])
\]
Assume that $n=[\Delta:\Lambda]$. If the genus of $X$ satisfies $g>(p-1)^{2}$,
then the $p$-gonal covering $X\rightarrow\widehat{\mathbb{C}}=X/H$ is unique,
\cite{A}, hence $\Lambda\vartriangleleft\Delta$ and there is $\varpi
:\Delta\rightarrow\Delta/\Gamma\cong G$ such that $\varpi^{-1}(H)=\varpi
^{-1}(\Lambda/\Gamma)=\Lambda$. Changing the indices in the canonical
generators $x_{1},...,x_{r}$ in such a way that $\varpi(x_{i}^{j}%
)\notin\Lambda/\Gamma$, for any $0<j<m_{i}$, $i=t+1,...,r$, but this condition
is not satisfied for $x_{1},...,x_{t}$, we obtain:
\[
\lbrack m_{1},...,m_{r}]=[ps_{1},...,ps_{t},m_{t+1},...,m_{r}].
\]
Given the signature of $\Lambda$, the fact that $n=[\Delta:\Lambda]$ implies
that the integers $s_{1},...,s_{t}$ divide to $n$ and $q=\frac{n}{s_{1}%
}+...+\frac{n}{s_{t}}$. Applying Riemann-Hurwitz formula we get:%
\begin{equation}
\frac{-2+q(1-\frac{1}{p})}{h-2+%
{\displaystyle\sum\nolimits_{i=1}^{t}}
(1-\frac{1}{s_{i}p})+%
{\displaystyle\sum\nolimits_{i=t+1}^{r}}
(1-\frac{1}{m_{i}})}=n \label{formula}%
\end{equation}
From the above formula and using $q=\sum\frac{n}{s_{i}}$ we obtain:%
\begin{align*}
n(h-2)  &  =-2+(%
{\displaystyle\sum\nolimits_{i=1}^{t}}
\frac{n}{s_{i}})(1-\frac{1}{p})-n%
{\displaystyle\sum\nolimits_{i=1}^{t}}
(1-\frac{1}{s_{i}p})-n%
{\displaystyle\sum\nolimits_{i=t+1}^{r}}
(1-\frac{1}{m_{i}})\\
&  <(%
{\displaystyle\sum\nolimits_{i=1}^{t}}
\frac{n}{s_{i}})(1-\frac{1}{p})-n%
{\displaystyle\sum\nolimits_{i=1}^{t}}
(1-\frac{1}{s_{i}p})=n%
{\displaystyle\sum\nolimits_{i=1}^{t}}
(\frac{1}{s_{i}}-1)\leq0
\end{align*}
Thus $h=1$ (note that $h>0$, since the sign in the signature of $\Delta$ is
$-$). Now formula (\ref{formula}) is equivalent to:%
\[
2-n+n%
{\displaystyle\sum\nolimits_{i=1}^{t}}
(1-\frac{1}{s_{i}})+n%
{\displaystyle\sum\nolimits_{i=t+1}^{r}}
(1-\frac{1}{m_{i}})=0
\]
where $s_{i}\geq1$ and $m_{i}\geq2$. From there it is easy to deduce that
there is at most one $s_{i}>1$. We have two possible cases:

i) $s_{i}=1$, for all $i$, in which case we have
\[%
{\displaystyle\sum\nolimits_{i=t+1}^{r}}
(1-\frac{1}{m_{i}})=1-\frac{2}{n},
\]
thus $r-t=1$ and $m_{t+1}=\frac{n}{2}$.

ii) There is one $i$ (we may suppose $i=t$) such that $s_{t}>1$, then
$s_{t}=\frac{n}{2}$, $s_{1}=...=s_{t-1}=1$ and $r=t$.

Therefore the possible signatures for $\Delta$ are:%
\begin{align}
&  i)\text{ }(1;-;[p,\overset{\frac{2(g+p-1)}{n(p-1)}}{...},p,\frac{n}%
{2}])\label{signatures}\\
&  ii)\text{ }(1;-;[p,\overset{\frac{2g}{n(p-1)}}{...},p,\frac{n}%
{2}p])\nonumber
\end{align}

Assume then that $\Delta$ is an NEC group with one of the signatures in
(\ref{signatures})\ and let
\[
l=\frac{2(g+p-1)}{n(p-1)}%
\]
in signature i) and%
\[
l=\frac{2g}{n(p-1)}%
\]
in signature ii).

We shall study the epimorphism%
\[
\theta:\Delta\rightarrow\Delta/\Lambda
\]
Let $d,x_{1},...,x_{l+1}$ be the generators of a canonical presentation of
$\Delta$. If $\pi:\Delta/\Gamma\rightarrow\Delta/\Lambda$ is the natural
projection, then $\pi\circ\varpi=\theta$, and since $\varpi(x_{i})\in
\Lambda/\Gamma,i=1,...,l,$ we have that $\theta(x_{1})=...=\theta(x_{l})=1$,
so the group $\Delta/\Lambda$ is generated by $\theta(d)$ and $\theta
(x_{l+1})$ ($=\theta(d)^{-2}$ by the long relation of NEC\ groups). Therefore
the group $\Delta/\Lambda$ is cyclic of order $n$ and $G\cong\Delta/\Gamma$ is
a normal extension of a cyclic group $C_{p}\cong\Lambda/\Gamma$ by a cyclic
group $C_{n}\cong\Delta/\Lambda$:%
\[
0\rightarrow C_{p}\rightarrow G\overset{\pi}{\rightarrow}C_{n}\rightarrow0
\]

The covering $X/H\cong\mathbb{D}/\Lambda\rightarrow X/G\cong\mathbb{D}/\Delta$
is a cyclic covering with automorphism group $C_{n}\cong\Delta/\Lambda
=\left\langle \theta(d)\right\rangle =\left\langle \widetilde{y}\right\rangle
$. The lift $y\in G$ of the automorphism $\widetilde{y}$ to $X=\mathbb{D}%
/\Gamma$ must have order $n$ or $np$. In the first case $G$ is a semidirect
product $C_{n}\ltimes_{r}H$ with presentation (\ref{PresenMeta}) and in the
second case $G$ is cyclic ($\cong C_{np}$).

Suppose the signature of $\Delta$ is $ii)$ of (\ref{signatures}).

Now $\varpi(x_{l+1})$ is an element of order $q=\frac{n}{2}p$ in $G$ and
$\varpi(d)$ has order either $n$ or $np$.

If the order of $\varpi(d)$ is $np$, then $G\cong C_{np}$.

In the case $\varpi(d)$ has order $n$, $G$ is generated by $x,y$ where
$\varpi(d)=y$, $\varpi(x_{l})=x$, and the following relations hold:%
\[
x^{p}=1;y^{n}=1;y^{-1}xy=x^{r}%
\]

Since there is an element of order $\frac{n}{2}p$ in $G$ then $\Delta
^{+}/\Gamma$ is cyclic of order $\frac{n}{2}p$. Since $y^{2}\in\Delta
^{+}/\Gamma$, we have $y^{-2}xy^{2}=x=x^{r^{2}}$, and then $r=1$ or $p-1$.

In all the cases, the surface $X$ being pseudo-real, we know by Proposition
\ref{4divides}, that $4$ divides $n$. Finally, either $\frac{2(g+p-1)}%
{n(p-1)}$ or $\frac{2g}{n(p-1)}$ is an integer, thus $g$ is even.
\end{proof}

\bigskip

\textbf{Remarks.} 1. Condition $g>(p-1)^{2}$ in the above theorem may be
replaced by the assumption that the cyclic group $H$ is normal in $G$. The
$n$-gonal surfaces whose full automorphism group $G$ is not the normalizer of
$H$ have been studied in \cite{BW}, \cite{W}

2. Theorem \ref{mainthm} remains valid if $G$ is a group of automorphisms (not
necessarily the full automorphism group)\ containing a $p$-gonal automorphism
and an anticonformal automorphism of order $>2$.

\bigskip

\begin{theorem}
\label{conditionssemidirectgeneral} For each prime $p>2$ and each $n$, such
that $4$ divides $n$, there exists a cyclic $p$-gonal pseudo-real Riemann
surface of genus $g$ with full automorphism group isomorphic to $C_{p}%
\rtimes_{r}C_{n}$ with presentation (\ref{PresenMeta}) and $1<r<p-1$ if and
only if $\frac{2(g+p-1)}{n(p-1)}$ is an integer $>1$.
\end{theorem}

\begin{proof}
Assume that $l=\frac{2(g+p-1)}{n(p-1)}$ is an integer $>1$. Let $\Delta$ be a
maximal NEC group with signature $(1;-;[p,\overset{l}{...},p,\frac{n}{2}])$.
We define
\[
\varpi:\Delta\rightarrow C_{n}\ltimes_{r}C_{p}=\left\langle x,y:x^{p}%
=1;y^{n}=1;y^{-1}xy=x^{r},r^{n}\equiv1\operatorname{mod}p\right\rangle
\]
by%
\begin{align*}
\varpi(x_{2i+1})  &  =x,\varpi(x_{2i+2})=x^{-1},i=0,...,(l-2)/2,\\
\varpi(x_{l+1})  &  =y^{2},\varpi(d)=y^{-1}\text{, if }l\text{ is even}%
\end{align*}
and%
\begin{align*}
\varpi(x_{2i+1})  &  =x,\varpi(x_{2i+2})=x^{-1},i=0,...,(l-3)/2,\\
\varpi(x_{l-2})  &  =x,\varpi(x_{l-1})=x,\varpi(x_{l})=x^{-2}\\
\varpi(x_{l+1})  &  =y^{2},\varpi(d)=y^{-1}\text{, if }l\text{ is odd, }l>1
\end{align*}
Every element of $C_{n}\ltimes_{r}C_{p}$ can be expressed as $x^{\alpha
}y^{\beta}$, by using the relation $y^{-1}xy=x^{r}$. Note that $\varpi
(\Delta^{+})=\varpi\left\langle x_{1},...,x_{l},x_{l+1},d^{2}\right\rangle
=\left\langle x,y^{2}\right\rangle $, $4$ divides $n$ and $n/2$ is an even
number. Thus the elements $x^{\alpha}y^{n/2}$ of order $2$ of the group
$C_{n}\ltimes_{r}C_{p}$, that can be written as $x^{\alpha}y^{n/2}$,
correspond to orientation preserving automorphisms. The epimorphism $\varpi$
cannot be extended since $\Delta$ is maximal and $\mathbb{D}/\ker\varpi$ is
the pseudo-real Riemann surface that we are looking for.

If $l=1$ by Theorem 2.4.7\ of \cite{BEGG} the group $\Delta$ is contained in
an NEC\ group $\Omega$ with signature $(0;+;[2];\{(p,\frac{n}{2})\}$. Let
\begin{gather*}
C_{2}\ltimes(C_{n}\ltimes_{r}C_{p})=\\
\left\langle s,x,y:s^{2}=1,x^{p}=1;y^{n}=1;y^{-1}xy=x^{r};sxs=x^{-1}%
,sys=y^{-1}\right\rangle
\end{gather*}
We may assume that $\varpi:\Delta\rightarrow C_{n}\ltimes_{r}C_{p}$ is given
by $\varpi(d)=y,\varpi(x_{1})=x,\varpi(x_{2})=y^{-2}$, where $d,x_{1}$ and
$x_{2}$ are generators of a canonical presentation of $\Delta$. The
epimorphism $\varpi$ is the restriction to $\Delta$ of
\begin{align*}
\theta^{\ast}  &  :\Omega\rightarrow C_{2}\ltimes(C_{n}\ltimes C_{p})\\
\theta^{\ast}(x_{1}^{\prime})  &  =sx^{\alpha}y^{-1}\text{, }\theta^{\ast
}(c_{0}^{\prime})=s\text{, }\theta^{\ast}(c_{1}^{\prime})=sx,\text{ }%
\theta^{\ast}(c_{2}^{\prime})=sxy^{-2},\text{ }\theta^{\ast}(e)=sx^{\alpha
}y^{-1}%
\end{align*}
where $\alpha\equiv\frac{p+1}{2}r^{n-1}\operatorname{mod}p$. Hence there are
anticonformal involutions in $\mathrm{Aut}(\mathbb{D}/\ker\varpi)$ and
$\mathbb{D}/\ker\varpi$ is not a pseudo-real Riemann surface.
\end{proof}

\bigskip

\begin{theorem}
\label{conditionsrealizability} For each prime $p>2$ and each $n$, such that
$4$ divides $n$, there exists a cyclic $p$-gonal pseudo-real Riemann surface
of genus $g$, with full automorphism group isomorphic to $C_{np}$ if and only
if either $\frac{2(g+p-1)}{n(p-1)}$ or $\frac{2g}{n(p-1)}$ are integers $>1$
and $\gcd(p,n/2)=1$.
\end{theorem}

\begin{proof}
If either $\frac{2(g+p-1)}{n(p-1)}$ or $\frac{2g}{n(p-1)}$ are integers $>1$
with $\gcd(p,n/2)=1$, we may consider maximal NEC\ groups with signatures
either
\[
(1;-;[p,\overset{\frac{2(g+p-1)}{n(p-1)}}{...},p,\frac{n}{2}])\text{ or
}(1;-;[p,\overset{\frac{2g}{n(p-1)}}{...},p,\frac{n}{2}p])
\]
respectively. The surface Fuchsian groups uniformizing the cyclic $p$-gonal
pseudo-real Riemann surfaces having cyclic automorphism group are the kernel
of the epimorphisms given by Theorem 4 of \cite{Et}. 

Assume that there is a cyclic $p$-gonal pseudo-real Riemann surface $X$ of
genus $g$, with full automorphism group $C_{np}$. By Theorem 4, there is a
maximal NEC\ group $\Delta$ with signature as given in (\ref{signatures})
(then either $\frac{2(g+p-1)}{n(p-1)}$ or $\frac{2g}{n(p-1)}$ are integers)
and there is an epimorphism $\varpi:\Delta\rightarrow C_{np}$ such that
$\ker\varpi$ uniformizes $X$. Condition ii) of Theorem 4 of \cite{Et} implies
$\gcd(p,n/2)=1$.

For the first possible signature of $\Delta$, assume $\frac{2(g+p-1)}%
{n(p-1)}=1$. The epimorphism $\varpi:\Delta\rightarrow C_{np}=\left\langle
t\right\rangle $, where $\Delta$ has signature $(1;-;[p,\frac{n}{2}])$ is
given by $\varpi(d)=t,\varpi(x_{1})=t^{2m},\varpi(x_{2})=t^{-2m-2}$, where
$d,x_{1}$ and $x_{2}$ are generators of a canonical presentation of $\Delta$,
and $0<m<p$ (note that $x_{1}$ is an orientation preserving transformation and
then $\varpi(x_{1})$ must be an even power of $t$). By Theorem 2.4.7\ of
\cite{BEGG} the group $\Delta$ is contained in an NEC\ group $\Omega$ with
signature $(0;+;[2];\{(p,\frac{n}{2}p)\})$ and the epimorphism $\varpi$ is the
restriction to $\Delta$ of
\begin{align*}
\theta^{\ast}  &  :\Omega\rightarrow D_{np}=\left\langle s,t:s^{2}%
=t^{np}=1,sts=t^{-1}\right\rangle \\
\theta^{\ast}(x_{1}^{\prime})  &  =t^{2m+1}\text{, }\theta^{\ast}%
(c_{0}^{\prime})=t^{2m}s\text{, }\theta^{\ast}(c_{1}^{\prime})=s,\text{
}\theta^{\ast}(c_{2}^{\prime})=st^{-2m-2}%
\end{align*}
where $x_{1}^{\prime},c_{0}^{\prime},c_{1}^{\prime},c_{2}^{\prime}$ is a set
of generators of a canonical presentation of $\Omega$. Then $\ker\theta^{\ast
}=\ker\varpi=\Gamma$ and $\mathrm{Aut}^{\pm}(\mathbb{D}/\Gamma)$ contains
$D_{np}$ with anticonformal involutions. Hence $\mathbb{D}/\Gamma$ is not
pseudo-real. Thus $\frac{2(g+p-1)}{n(p-1)}>1$.

The argument for the second possible signature of $\Delta$ with $\frac
{2g}{n(p-1)}=1$ is similar.
\end{proof}

\begin{theorem}
For each prime $p>2$ and each $n$, such that $4$ divides $n$, there exists a
cyclic $p$-gonal pseudo-real Riemann surface of genus $g$ with full
automorphism group isomorphic to $C_{p}\rtimes_{r}C_{n}$ with presentation
(\ref{PresenMeta}) and $r=1$ or $p-1$ if and only if either $\frac
{2(g+p-1)}{n(p-1)}$ is an integer $>1$ or $\frac{2g}{n(p-1)}$ is an integer
$>1$ and $\gcd(p,n/2)=1$.
\end{theorem}

\begin{proof}
If either $\frac{2(g+p-1)}{n(p-1)}$ or $\frac{2g}{n(p-1)}$ are integers $>1$,
we may consider maximal NEC\ groups with signatures either
\[
(1;-;[p,\overset{\frac{2(g+p-1)}{n(p-1)}}{...},p,\frac{n}{2}])\text{ or
}(1;-;[p,\overset{\frac{2g}{n(p-1)}}{...},p,\frac{n}{2}p])
\]
respectively. For the first case considering the epimorphism defined in
Theorem \ref{conditionssemidirectgeneral} we obtain the surface that we are
looking for. If the signature is $(1;-;[p,\overset{l}{...},p,\frac{n}{2}p])$,
with $l=\frac{2g}{n(p-1)}$, and $\gcd(p,n/2)=1$, then we define:%
\[
\varpi:\Delta\rightarrow C_{n}\ltimes_{r}C_{p}=\left\langle x,y:x^{p}%
=1;y^{n}=1;y^{-1}xy=x^{r}\right\rangle
\]
by%
\begin{align*}
\varpi(x_{2i+1})  &  =x,\varpi(x_{2i+2})=x^{-1},i=0,...,(l-3)/2,\\
\varpi(x_{l})  &  =x,\varpi(x_{l+1})=x^{-1}y^{2},\varpi(d)=y^{-1}\text{, if
}l\text{ is odd}%
\end{align*}
and%
\begin{align*}
\varpi(x_{2i+1})  &  =x,\varpi(x_{2i+2})=x^{-1},i=0,...,(l-4)/2,\\
\varpi(x_{l-1})  &  =x,\varpi(x_{l})=x\\
\varpi(x_{l+1})  &  =x^{-2}y^{2},\varpi(d)=y^{-1}\text{, if }l\text{ is even}%
\end{align*}
Note that $x^{-1}y^{2}$ and $x^{-2}y^{2}$ have order $\frac{n}{2}p$ as
consequence of the condition $\gcd(p,n/2)=1$.

For the first possible signature of $\Delta$, if we assume $\frac
{2(g+p-1)}{n(p-1)}=1$, the argument in Theorem
\ref{conditionssemidirectgeneral} shows that the surfaces in these conditions
are not pseudoreal. Assume now that $\Delta$ has the second possible signature
with $l=1$. We may assume that the epimorphism $\varpi:\Delta\rightarrow
C_{p}\rtimes_{r}C_{n}$, where $\Delta$ has signature $(1;-;[p,\frac{n}{2}p])$
is given by $\varpi(d)=y,\varpi(x_{1})=x,\varpi(x_{2})=x^{-1}y^{-2}$, where
$d,x_{1}$ and $x_{2}$ are generators of a canonical presentation of $\Delta$.
By Theorem 2.4.7\ of \cite{BEGG} the group $\Delta$ is contained in an
NEC\ group $\Omega$ with signature $(0;+;[2];\{(p,\frac{n}{2}p)\})$. If $r=1$,
the epimorphism $\varpi$ is the restriction to $\Delta$ of%
\begin{align*}
\theta^{\ast}  &  :\Omega\rightarrow C_{2}\ltimes(C_{n}\ltimes C_{p})\\
\theta^{\ast}(x_{1}^{\prime})  &  =sy^{-1}\text{, }\theta^{\ast}(c_{0}%
^{\prime})=s\text{, }\theta^{\ast}(c_{1}^{\prime})=sx,\text{ }\theta^{\ast
}(c_{2}^{\prime})=sy^{-2},\theta^{\ast}(e^{\prime})=sy^{-1}%
\end{align*}
where%
\begin{gather*}
C_{2}\ltimes(C_{n}\ltimes_{r}C_{p})=\\
\left\langle s,x,y:s^{2}=1,x^{p}=1;y^{n}=1;y^{-1}xy=x^{r};sxs=x^{-1}%
,sys=y^{-1}\right\rangle ,\\
r=1,p-1
\end{gather*}

Hence the surface uniformized by $\ker\varpi$ has anticonformal involutions.
\end{proof}

\section{Cyclic $p$-gonal pseudo-real Riemann surfaces with automorphism group
of maximal order for a fixed genus.}

For each type of automorphism groups of pseudo-real $p$-gonal Riemann surfaces
of genus $g$, next proposition determines its maximal order:

\begin{proposition}
Let $p$ be a prime $>2$ and $g>(p-1)^{2}$. Assume that $p-1$ divides $g$.

1. If $\frac{g}{p-1}\equiv3\operatorname{mod}4$ there are pseudo-real cyclic
$p$-gonal Riemann surfaces of genus $g$ with full automorphism group of order
$\frac{p(g+p-1)}{p-1}$ and this order is maximal for pseudo-real cyclic
$p$-gonal Riemann surfaces of genus $g$. The full automorphism group is either
isomorphic to a semidirect product of cyclic groups or, if $\gcd
(p,\frac{g+p-1}{2(p-1)})=1$, may be isomorphic to a cyclic group.

2. If $\frac{g}{p-1}\equiv0\operatorname{mod}4$ and $\gcd(p,\frac
{g+p-1}{2(p-1)})=1$ there are pseudo-real cyclic $p$-gonal Riemann surfaces of
genus $g$ with full automorphism group of order $\frac{pg}{(p-1)}$ and this
order is maximal for pseudo-real cyclic $p$-gonal Riemann surfaces of genus
$g$. In this case the full automorphism group is either cyclic or
$C_{p}\rtimes_{r}C_{n}$ with $r=1$ or $p-1.$
\end{proposition}

\begin{proof}
Assume that $X$ is a pseudo-real cyclic $p$-gonal Riemann surface of genus $g$
with $X/\mathrm{Aut}(X)$ uniformized by an NEC group where signature $i)$ of
(\ref{signatures}) in theorem \ref{mainthm} and then
\[
2\leq\frac{2(g+p-1)}{n(p-1)}%
\]
or equivalently:
\[
n\leq\frac{g}{(p-1)}+1
\]

When $p-1$ divides $g$ we may have $n=\frac{g}{(p-1)}+1$. Hence there are
surfaces with automorphism groups of order $np=\frac{p(g+p-1)}{p-1}$ and this
order is maximal. Note that $\frac{g}{p-1}\equiv3\operatorname{mod}4$, since
$4$ must divide $n$. By theorem \ref{mainthm} the automorphism group is either
isomorphic to $C_{p}\rtimes_{r}C_{n}$ or, if $\gcd(p,\frac{g+p-1}{2(p-1)})=1$,
may be isomorphic to $C_{np}$.

Now assume that $X$ is a pseudo-real cyclic $p$-gonal Riemann surface of genus
$g$ with $X/\mathrm{Aut}(X)$ uniformized by an NEC group with signature $ii)$
of (\ref{signatures}).

Since $4$ must divide $n$, if $\frac{g}{p-1}\equiv0\operatorname{mod}4$ and
$\gcd(p,\frac{g}{2(p-1)})=1$, there are surfaces with full automorphism group
isomorphic to a maximal order group (the order is $\frac{pg}{2(p-1)}$). Note
that when $\frac{g}{p-1}\equiv0\operatorname{mod}4$, $4$ does not divide
$\frac{2(g+p-1)}{p-1}$ and then case 1 of theorem \ref{mainthm} is not
possible. By theorem \ref{mainthm} case 2, the automorphism group is
isomorphic to $C_{np}$ or $C_{p}\rtimes_{r}C_{n}$ with $r=1,p-1$.
\end{proof}

\section{Cyclic $p$-gonal pseudo-real Riemann surfaces in the moduli space.}

In this section we apply the previous results to the study of branch locus of
moduli spaces.

First we define:

$\mathcal{M}_{g}^{PR}(n,p)=\{X:X$ is a cyclic $p$-gonal pseudo-real Riemann
surface of genus $g$ and with full automorphism group of order $np\}.$

Let $\mathcal{M}_{g}$ be the moduli space of Riemann surfaces of genus $g$,
$\mathbb{T}_{g}$ be the corresponding Teichm\"{u}ller space and
\[
\pi:\mathbb{T}_{g}\rightarrow\mathbb{T}_{g}/\mathrm{Mod}_{g}=\mathcal{M}_{g}%
\]
be the natural projection.

\begin{proposition}
Let $p$ be a prime $>2$ and $g\geq2$. Assume that $p-1$ divides $g$.

1. If $\frac{g}{p-1}\equiv3\operatorname{mod}4$ then the set $\mathcal{M}%
^{PR}(\frac{g+p-1}{p-1},p)$ is a $2$-manifold in the branch locus of the
orbifold $\mathcal{M}_{g}$.

2. If $\frac{g}{p-1}\equiv0\operatorname{mod}4$ and $\gcd(p,\frac
{g+p-1}{2(p-1)})=1$ then the set $\mathcal{M}^{PR}(\frac{g}{p-1},p)$ is a
$2$-manifold in the branch locus of the orbifold $\mathcal{M}_{g}$.
\end{proposition}

\begin{proof}
First we consider the case $\frac{g}{p-1}\equiv3\operatorname{mod}4$ and we
write $n=\frac{g}{p-1}+1$. Let $\mathbb{T}_{(1;-;[p,p,n/2])}$ be the
Teichm\"{u}ller space for NEC groups with signature $(1;-;[p,p,\frac{n}{2}])$
and $M$ be the subset of maximal groups of $\mathbb{T}_{(1;-;[p,p,n/2])}$.

Consider the set $S=\{(\sigma,i):$ $\sigma$ is a signature of NEC groups and
$i$ is an inclusion of a group of signature $(1;-;[p,p,\frac{n}{2}])$ in a
group of signature $\sigma\}$. Using Riemann-Hurwitz and canonical
presentations of NEC\ groups it is possible to show that the set $S$ is
finite. Each element $(\sigma,i)\in S$ produces an embedding $i_{\ast
}:\mathbb{T}_{\sigma}\rightarrow\mathbb{T}_{(1;-;[p,p,n/2])}$ and $i_{\ast
}(\mathbb{T}_{\sigma})$ is closed (see \cite{MS}). Then $M=\mathbb{T}%
_{(1;-;[p,p,n/2])}-%
{\textstyle\bigcup\nolimits_{(\sigma,i)\in S}}
i_{\ast}(\mathbb{T}_{\sigma})$ is an open set.

Let $G$ be $C_{p}\rtimes_{r}C_{n}$, $\mathcal{D}$ be an abstract group with
presentation
\[
\left\langle d,x_{1},x_{2},x_{3}:d^{2}x_{1}x_{2}x_{3}=1,x_{1}^{p}=x_{2}%
^{p}=x_{3}^{n/2}=1\right\rangle
\]
and $\varpi:\mathcal{D}\rightarrow G$ be an epimorphism as considered in the
proof of theorem \ref{mainthm}. If $\Delta$ has signature $(1;-;[p,p,\frac
{n}{2}])$ and $\eta:\Delta\rightarrow\mathcal{D}$ is the isomorphism given by
the canonical presentation then $\ker\varpi\circ\eta$ uniformizes a
pseudo-real cyclic $p$-gonal Riemann surface of genus $g$. The epimorphism
$\varpi$ induces an embedding $\mathbb{T}_{(1;-;[p,p,n/2])}\overset{\varpi
^{\ast}}{\rightarrow}\mathbb{T}_{g}$ defined by $[\Delta]\mapsto\lbrack
\ker\varpi\circ\eta]$ (see \cite{MS}). For two embeddings $\varpi_{1}^{\ast}$
and $\varpi_{2}^{\ast}$ defined by two epimorphisms $\varpi_{1}$ and
$\varpi_{2}$, we have that $\varpi_{1}^{\ast}(M)\cap\varpi_{2}^{\ast
}(M)=\varnothing$, since a possible point in the intersection would admit two
different actions of maximal order which is absurd. Finally the action of the
modular group $\mathrm{Mod}_{g}$ is fixed point free on $\pi^{-1}%
(\mathcal{M}^{PR}(n,p))$, which implies the projection
\[
\pi:\pi^{-1}(\mathcal{M}^{PR}(n,p))=\cup_{\varpi}\varpi^{\ast}(M)\rightarrow
\mathcal{M}^{PR}(n,p)
\]
is a local homeomorphism. Hence $\mathcal{M}^{PR}(n,p)$ is a manifold of
(real) dimension $2$ in $\mathcal{M}_{g}$.

The argument for the second case is similar.


\end{proof}

\end{document}